\documentclass[a4paper,twoside,fleqn]{article}
\usepackage{CJK} 
\usepackage{amscd}
\usepackage{amsfonts}
\usepackage[dvips]{graphics}
\usepackage{amssymb}
\usepackage[reqno]{amsmath}
\usepackage{amsthm}
\usepackage{makeidx}
\usepackage{color}
\usepackage{graphicx}
\usepackage{subfigure}
\usepackage{multicol}
\usepackage{fancyhdr}
\usepackage[colorlinks,linkcolor=blue,anchorcolor=blue,citecolor=blue]{hyperref}

 \newtheorem{theorem}{\sc\bf Theorem}[section]
 
 \newtheorem{corollary}[theorem]{\sc\bf Corollary}
 \newtheorem{lemma}[theorem]{\sc\bf Lemma}
 
 \newtheorem{definition}[theorem]{\sc\bf Definition}
 \newtheorem{remark}[theorem]{\sc\bf Remark}
 \newtheorem{question}[theorem]{\sc\bf Question}
 \newtheorem{example}[theorem]{\sc\bf Example}
  \newtheorem{observation}[theorem]{\sc\bf Observation}
  \numberwithin{equation}{section}

\makeatletter
\def\@cite#1#2{#1\if@tempswa , #2\fi}
\makeatother

\large

\title{{\bf On $(n,k)$-quasi-$*$-paranormal operators}
\thanks{This work has been supported by National Natural Science Foundation
of China (11171066), Specialized Research Fund for the Doctoral
Program of Higher Education (2010350311001, 20113503120003), Natural
Science Foundation of Fujian Province (2011J05002, 2012J05003) and
Foundation of the Education Department of Fujian Province
(JB10042).} }

\author{Qingping \textsc{Zeng},\thanks{Email address: zqpping2003@163.com.}
 \ \  Huaijie \textsc{Zhong}
\\ \small (School of Mathematics and Computer Science, Fujian Normal
University, Fuzhou 350007, P.R. China) }

\begin{document}

\date{}
\maketitle

\large

\begin{quote}
 {\bf Abstract:} ~For nonnegative integers $n$ and $k$, we introduce in this paper a new class of $(n,k)$-quasi-$*$-paranormal operators
 satisfying
 $$||T^{1+n}(T^{k}x)||^{1/(1+n)}||T^{k}x||^{n/(1+n)} \geq
 ||T^*(T^{k}x)|| \makebox{\ for all } x \in H.$$
 This class includes the class of $n$-$*$-paranormal operators and the class of
 $(1,k)$-quasi-$*$-paranormal operators
 contains the class of $k$-quasi-$*$-class $A$ operators. We study basic properties of $(n,k)$-quasi-$*$-paranormal
 operators: (1) inclusion relations and examples; (2) a matrix
 representation; (3) joint (approximate) point spectrum and single valued extension property.

\end{quote}

\section{Introduction}

 \quad\,Let $L(H)$ stand for the $C^{*}$ algebra of all bounded linear
 operators on an infinite dimensional complex Hilbert space $H$. As
 an extension of normal operators, P. Halmos introduced the class of
 hyponormal operators (defined by $TT^{*} \leq T^{*}T$) [\cite{Halmos}]. Although there are still many interesting problems
 for hyponormal operators yet to solve (e.g., the invariant subspace
 problem), one of recent hot topics in operator theory is to study
 natural extensions of hyponormal operators. Below are some of these nonhyponormal operators. Recall that an operator $T \in L(H)$ is said
 to be

 $\bullet$ $*$-$class$ $A$ if $|T^{2}| \geq |T^*|^{2}$ (see
 [\cite{Duggal-Jeon-Kim}]).

 $\bullet$ $quasi$-$*$-$class$ $A$ if $T^{*}|T^{2}|T \geq
 T^{*}|T^*|^{2}T$ (see [\cite{Shen-Zuo-Yang}]).

 $\bullet$ $k$-$quasi$-$*$-$class$ $A$ if
$T^{*k}|T^{2}|T^{k} \geq
 T^{*k}|T^*|^{2}T^{k}$ (see [\cite{Mecheri-SM}]).

 $\bullet$ $*$-$paranormal$ if
$||T^{2}x||^{1/2}||x||^{1/2} \geq
 ||T^*x||$ for all $x \in H$ (see [\cite{Arora-Thukral}]).

$\bullet$ $quasi$-$*$-$paranormal$ if
$||T^{2}(Tx)||^{1/2}||Tx||^{1/2} \geq
 ||T^*(Tx)||$ for all $x \in H$ (see [\cite{Mecheri-AFA}]).

$\bullet$ $k$-$quasi$-$*$-$paranormal$ if
$||T^{2}(T^{k}x)||^{1/2}||T^{k}x||^{1/2} \geq
 ||T^*(T^{k}x)||$ for all $x \in H$.

 $\bullet$ $n$-$*$-$paranormal$ if
$||T^{1+n}x||^{1/(1+n)}||x||^{n/(1+n)} \geq
 ||T^*x||$ for all $x \in H$ (see [\cite{Lee-Lee-Rhoo}]).

\noindent Here and henceforth, $n,k$ denote nonnegative integers.

Clearly, if $k=1$ (respectively, $k=0$), then $k$-quasi-$*$-class
$A$ is precisely quasi-$*$-class $A$ (respectively, $*$-class $A$),
and $k$-quasi-$*$-paranormal is precisely quasi-$*$-paranormal
(respectively, $*$-paranormal); if $n=1$ (respectively, $n=0$), then
$n$-$*$-paranormal is precisely $*$-paranormal (respectively,
hyponormal). Moreover, $*$-class $A$ operators are $*$-paranormal
[\cite{Duggal-Jeon-Kim}, Theorem1.3], and $k$-quasi-$*$-class $A$
operators are $k$-quasi-$*$-paranormal (see Theorem \ref{2.1} below)
and contain $*$-class $A$.

As an extension of the classes of $n$-$*$-paranormal operators and
$k$-quasi-$*$-paranormal operators, the following definition
describes the class of operators we will study in this paper.

\begin {definition} \label{1.1}  An operator $T \in L(H)$ is said to be
$(n,k)$-$quasi$-$*$-$paranormal$ if
 $$||T^{1+n}(T^{k}x)||^{1/(1+n)}||T^{k}x||^{n/(1+n)} \geq
 ||T^*(T^{k}x)|| \makebox{ for all } x \in H.$$
\end{definition}

Clearly, if $n=1$ (respectively, $n=0$), then
$(n,k)$-quasi-$*$-paranormal is precisely $k$-quasi-$*$-paranormal
(respectively, $k$-quasihyponormal which is introduced in
[\cite{Campbell-Gupta}] and defined by $T^{*k}TT^{*}T^{k} \leq
T^{*k}T^{*}TT^{k}$); if $k=0$, then $(n,k)$-quasi-$*$-paranormal is
precisely $n$-$*$-paranormal.

In this study we give basic properties of
$(n,k)$-quasi-$*$-paranormal
 operators. In Section 2, we discuss some inclusion relations and examples related to $(n,k)$-quasi-$*$-paranormal
 operators. In Section 3, a matrix
 representation is obtained and a negative answer to a question posed by Mecheri [\cite{Mecheri-SM}] is given.
 In Section 4, we show that, for every $(n,k)$-quasi-$*$-paranormal operator
$T$, the nonzero points of its point spectrum and joint point
spectrum are identical, the nonzero points of its approximate point
spectrum and joint approximate point spectrum are identical. As a
corollary, it is also shown that $(n,k)$-quasi-$*$-paranormal
operators have the single valued extension property.

\section{Inclusion relations and examples}

  \quad\, Recall that an operator $T \in L(H)$ is said to be

 $\bullet$ normaloid if $||T||=r(T)$, where $r(T)$ is the spectral
 radius of $T$ (see [\cite{Furuta}]).

 $\bullet$ hereditarily normaloid if every part of $T$ is normaloid, where a part of $T$ means its restriction to a closed invariant subspace
 (see [\cite{Duggal-Djordjevic}]).

 $\bullet$ $n$-$paranormal$ if $||T^{1+n}x||^{1/(1+n)}||x||^{n/(1+n)} \geq
 ||Tx||$ for all $x \in H$ (see [\cite{Istratescu-Istratescu}]).

 $\bullet$ $(n,k)$-$quasiparanormal$ if
$||T^{1+n}(T^{k}x)||^{1/(1+n)}||T^{k}x||^{n/(1+n)} \geq
 ||T(T^{k}x)||$ for all $x \in H$ (see [\cite{Yuan-Ji}]).

 Clearly, $(n,0)$-quasiparanormal is precisely $n$-paranormal.

\begin {theorem} \label{2.1} The following assertions hold.

 $(1)$ If $T$ is $k$-quasi-$*$-class $A$,
then it is $k$-$quasi$-$*$-$paranormal$.

$(2)$ If $T$ is $(n,k+1)$-quasi-$*$-paranormal, then it is
$(n+1,k)$-quasiparanormal.

$(3)$ If $T$ is $(n,k)$-quasi-$*$-paranormal and $M$ is an invariant
subspace of $T$, then $T|_{M}$ is also $(n,k)$-quasi-$*$-paranormal.

$(4)$ If $T$ is $(n,k)$-quasi-$*$-paranormal, then it is
$(n,k+1)$-quasi-$*$-paranormal.

$(5)$ If $T$ is $(n,0)$-quasi-$*$-paranormal or
$(n,1)$-quasi-$*$-paranormal, then it is hereditarily normaloid.

\end{theorem}

Theorem \ref{2.1}(5) generalizes [\cite{Mecheri-SM}, Theorem 2.6].
On the other hand, for $k \geq 2$, there exists an
$(n,k)$-quasi-$*$-paranormal but not normaloid operator (see Example
\ref{2.3}(4) below).

\begin{proof} (1) For all $x \in H$, we have
 $$(T^{*k}|T^*|^{2}T^{k}x, x) = (T^{*k}TT^*T^{k}x, x) =(T^*T^{k}x, T^*T^{k}x)=||T^*T^{k}x||^{2}$$
 and by H\"{o}lder-McCarthy inequality [\cite{McCarthy}, Lemma 2.1],
    \begin{align*} \qquad\qquad\qquad
      (T^{*k}|T^{2}|T^{k}x, x)  &=  (|T^{2}|T^{k}x, T^{k}x) \\ & \leq (T^{*2}T^{2}T^{k}x, T^{k}x)^{1/2} ||T^{k}x||^{2(1-1/2)}  \\
       &= ||T^{k+2}x||||T^{k}x||.
     \end{align*}
Since $T$ is $k$-quasi-$*$-class $A$, $||T^{k+2}x||||T^{k}x|| \geq
||T^*(T^{k}x)||^{2}$, for all $x \in H$. Hence $T$ is
$k$-quasi-$*$-paranormal.

 (2) Since $T$ is $(n,k+1)$-quasi-$*$-paranormal, we have, for all $x \in
 H$,
 \begin{align*} \qquad\qquad\qquad
      ||T^{k+1}x||^{2n+2}  &=  (T^{k+1}x, T^{k+1}x)^{n+1} \\ &= (T^*T^{k+1}x,
      T^{k}x)^{n+1} \\
        & \leq ||T^*T^{k+1}x||^{n+1}||T^{k}x||^{n+1} \\  &\leq ||T^{1+n}T^{k+1}x||||T^{k+1}x||^{n}||T^{k}x||^{n+1}.
     \end{align*}
Therefore, $||T(T^{k}x)||^{n+2} \leq ||T^{n+2}(T^{k}x)||
||T^{k}x||^{n+1}$, for all $x \in
 H$. Hence $T$ is
$(n+1,k)$-quasiparanormal.

 (3) It is clear.

 (4) It follows by taking $x=Tz$ in the definition.

 (5) By (3) and (4), it needs only to show that $T$ is normaloid when $T$ is $(n,1)$-quasi-$*$-paranormal. By (2), we
 have that $T$ is $(n+1,0)$-quasiparanormal, that is, $T$ is
 $(n+1)$-paranormal. It then follows form [\cite{Istratescu-Istratescu}, Theorem 1] that $T$ is normaloid.
\end{proof}

In order to establish the proper inclusion relation among the class
of $(n,k)$-quasi-$*$-paranormal operators and that of
$(n,k+1)$-quasiparanormal operators, the following lemma is needful.

\begin {lemma} \label{2.2}
  An operator $T \in L(H)$ is $(n,k)$-quasi-$*$-paranormal if and
  only if \begin{equation} \label{eq 2.1} \qquad \qquad
T^{*k} T^{*(1+n)}T^{1+n}T^{k} - (1+n) \mu^{n} T^{*k}TT^{*}T^{k} + n
 \mu^{1+n}T^{*k}T^{k} \geq 0
 \end{equation}
 for all $\mu > 0$.
\end{lemma}

This is a generalization of [\cite{Patel}, Theorem 4.14].

\begin{proof} The proof is similar to [\cite{Yuan-Ji}, Lemma 2.2] and [\cite{Yuan-Gao}, Lemma 2.2]. Let $T$ be
$(n,k)$-quasi-$*$-paranormal. It then follows form the weighted
arithmetic-geometric mean inequality that
    \begin{align*} \qquad \qquad\qquad
       &\quad \ \frac{1}{1+n}(\mu^{-n}|T^{1+n}|^{2}T^{k}x,T^{k}x) + \frac{n}{1+n}(\mu T^{k}x,T^{k}x)
        \\ &
       \geq (\mu^{-n}|T^{1+n}|^{2}T^{k}x,T^{k}x)^{\frac{1}{1+n}} (\mu T^{k}x,T^{k}x)^{\frac{n}{1+n}}
        \\ & = (|T^{1+n}|^{2}T^{k}x,T^{k}x)^{\frac{1}{1+n}} (T^{k}x,T^{k}x)^{\frac{n}{1+n}}
        \\ & \geq (|T^{*}|^{2}T^{k}x, T^{k}x) = (TT^{*}T^{k}x, T^{k}x).
     \end{align*}

Conversely, by (\ref{eq 2.1}), we have
\begin{equation} \label{eq 2.2}
(T^{*k} T^{*(1+n)}T^{1+n}T^{k}x,x) - (1+n) \mu^{n}
(T^{*k}TT^{*}T^{k}x,x) + n
 \mu^{1+n}(T^{*k}T^{k}x,x) \geq 0
 \end{equation}
 for all $x \in H.$
If $|T^{1+n}|^{2}T^{k}x,T^{k}x)=0,$ multiplying (\ref{eq 2.2}) by
$\mu^{-n}$ and letting $\mu \rightarrow 0$ we have
$(TT^{*}T^{k}x,T^{k}x)=0,$ thus
$$||T^{1+n}(T^{k}x)|| ||T^{k}x||^{n} \geq
 ||T^*(T^{k}x)||^{1+n}.$$
If $|T^{1+n}|^{2}T^{k}x,T^{k}x)>0,$ putting
$$\mu = \Big{(}\frac{(|T^{1+n}|^{2}T^{k}x, T^{k}x) }{(T^{k}x, T^{k}x) } \Big{)}^{\frac{1}{1+n}} $$
in (\ref{eq 2.2}) we have
$$(|T^{1+n}|^{2}T^{k}x, T^{k}x)^{1/(1+n)}(T^{k}x, T^{k}x)^{n/(1+n)} \geq
 (TT^*T^{k}x, T^{k}x).$$ So, $T$ is $(n,k)$-quasi-$*$-paranormal.
\end{proof}

To illustrate the result established in Theorem \ref{2.1}, we
consider the following

\begin{example} \label{2.3} \begin{upshape}
(1) An example of $*$-paranormal but not $*$-class $A$ operator.

Similar to an argument of Ando [\cite{Ando}], Duggal et al. showed
in [\cite{Duggal-Jeon-Kim}, p960] that there exists a $*$-paranormal
operator $T \in L(K)$ such that $T \otimes T$ is not $*$-paranormal:
$$ T:=T_{A,B}= \left(
                                                         \begin{array}{cccccc}
                                                           0 & 0 & 0 & 0 & 0 & \cdots \\
                                                           A & 0 & 0 & 0 & 0 & \cdots  \\
                                                           0 & B & 0 & 0 & 0 & \cdots  \\
                                                           0 & 0 & B & 0 & 0 & \cdots  \\
                                                           0 & 0 & 0 & B & 0 & \cdots  \\
                                                           \vdots & \vdots & \vdots & \vdots & \vdots & \cdots  \\
                                                         \end{array}
                                                       \right)
                                                       \makebox{\ on \ } K=\bigoplus\limits_{n=1}^{\infty}H,$$
where $\dim H=2,$ $A=\left(
                        \begin{array}{cc}
                          1 & 1 \\
                          1 & 2 \\
                        \end{array}
                      \right)^{1/2}
$ and $B=\left(
                        \begin{array}{cc}
                          1 & 2 \\
                          2 & 8 \\
                        \end{array}
                      \right)^{1/4}.$ From [\cite{Duggal-Jeon-Kim}, Theorem 3.2], it follows that $T$
is not $*$-class $A$. Otherwise, $T \otimes T$ is $*$-class $A$,
which is impossible.

(2) An example of $(n+1,k)$-quasiparanormal but not
$(n,k+1)$-quasi-$*$-paranormal operator, for nonnegative integers
$n,k$.

Let $U$ be the unilateral right shift operator on
$l_{2}(\mathbb{N})$ with the canonical orthogonal basis
$\{e_{m}\}_{m=1}^{\infty}$ defined by $Ue_{m}=e_{m+1}$ for all $m
\in \mathbb{N}$. Put $$T=\left(
                        \begin{array}{cc}
                          1 & \alpha e_{1} \otimes e_{1}  \\
                          0 & U+1 \\
                        \end{array}
                      \right) \makebox{\ on \ } H=\mathbb{C}e_{1} \oplus
                      l_{2}(\mathbb{N}).
$$ Uchiyama proved in [\cite{Uchiyama}] that $T$ is paranormal for
$0 < \alpha < \frac{1}{4}$ and $\mathrm{ker}(T-1)=\mathbb{C}e_{1}
\oplus \{0\}$ does not reduce $T$. By [\cite{Yuan-Ji}, Theorem 2.1],
we have that $T$ is $(n+1,k)$-quasiparanormal for all nonnegative
integers $n,k$. But, by Theorem \ref{4.1} below, $T$ is not
$(n,k+1)$-quasi-$*$-paranormal for all nonnegative integers $n,k$,
because $\mathrm{ker}(T-1)$ does not reduce $T$.

(3) An example of $(n,k+1)$-quasi-$*$-paranormal but not
$(n,k)$-quasi-$*$-paranormal operator, for nonnegative integer $n$
and positive integer $k \geq 1$.

Given a wight sequence $\{w_{m}\}_{m=1}^{\infty}$ of bounded and
positive numbers, let $T$ be the unilateral weighted right shift
operator on $l_{2}(\mathbb{N})$ with the canonical orthogonal basis
$\{e_{m}\}_{m=1}^{\infty}$ defined by $Te_{m}=w_{m}e_{m+1}$ for all
$m \in \mathbb{N}$. A routine calculation show that (see
[\cite{Yuan-Ji}, Example 2.3]), $TT^{*}=0 \oplus w_{1}^{2} \oplus
w_{2}^{2} \oplus \cdots ,$ and for each positive integer $m$,
\begin{equation} \label{eq 2.3} T^{*m}T^{m}=w_{1}^{2}\cdots
w_{m}^{2} \oplus w_{2}^{2}\cdots w_{m+1}^{2} \oplus w_{3}^{2}\cdots
 w_{m+2}^{2}
\oplus \cdots  \makebox{\quad on \ } l_{2}(\mathbb{N}),
\end{equation}
\begin{equation} \label{eq 2.4}
T^{*m}TT^{*}T^{m}=w_{1}^{2}\cdots w_{m}^{2}w_{m}^{2} \oplus
w_{2}^{2}\cdots w_{m+1}^{2}w_{m+1}^{2} \oplus w_{3}^{2}\cdots
 w_{m+2}^{2}w_{m+2}^{2}
\oplus \cdots  \makebox{ \ on \ } l_{2}(\mathbb{N}).
\end{equation}
By (\ref{eq 2.3}), (\ref{eq 2.4}) and Lemma \ref{2.2}, $T$ is
$(n,k)$-quasi-$*$-paranormal if and only if, for all $\mu > 0$,
\begin{equation*}  \quad
        \left\{ \begin{aligned}
           w_{k+1}^{2}\cdots w_{k+n+1}^{2}-(n+1)\mu^{n}w_{k}^{2}+n\mu^{n+1}  & \geq
           0,
         \\ w_{k+2}^{2}\cdots w_{k+n+2}^{2}-(n+1)\mu^{n}w_{k+1}^{2}+n\mu^{n+1}  & \geq
           0,
         \\ w_{k+3}^{2}\cdots w_{k+n+3}^{2}-(n+1)\mu^{n}w_{k+2}^{2}+n\mu^{n+1}  & \geq
           0,
         \\ \cdots\cdots. \qquad\qquad\qquad\qquad
         \end{aligned} \right.
\end{equation*}
It then follows that $T$ is $(n,k)$-quasi-$*$-paranormal if and only
if \begin{equation} \label{eq 2.5} \quad
        \left\{ \begin{aligned}
        w_{k}^{n+1} &  \leq w_{k+1}\cdots w_{k+n+1} ,
         \\ w_{k+1}^{n+1}  & \leq w_{k+2}\cdots w_{k+n+2},
         \\ w_{k+2}^{n+1}  & \leq w_{k+3}\cdots w_{k+n+3},
         \\ & \cdots\cdots.
         \end{aligned} \right.
\end{equation}
So, if $w_{k+1} \leq w_{k+2} \leq w_{k+3} \leq \cdots$ and $w_{k} >
w_{k+n+1}$, then $T$ is $(n,k+1)$-quasi-$*$-paranormal but not
$(n,k)$-quasi-$*$-paranormal.

(4)  An example of $(n,k)$-quasi-$*$-paranormal but not normaloid
operator, for $k \geq 2$.

Let $T$ be the unilateral weighted right shift operator as in (3)
and let
$$w_{1} > w_{2} = \cdots =w_{k}=w_{k+1}=  \cdots.$$ Then it is
obvious that $||T||=w_{1}$ and $$r(T)=\lim\limits_{m \rightarrow
\infty}||T^{m}||^{1/m} = w_{2},$$ therefore $T$ is not normaloid. By
(\ref{eq 2.5}), we have $T$ is $(n,k)$-quasi-$*$-paranormal.

(5)  An example of normaloid but not $(n,1)$-quasi-$*$-paranormal
operator, for all nonnegative integer $n$.

Let $S$ be an operator on $l_{2}(\mathbb{N})$ with the canonical
orthogonal basis $\{e_{m}\}_{m=1}^{\infty}$ defined by
 $$Se_{m}=\left\{ \begin{aligned}
          e_{m+1} & \quad \makebox{if} \, \,  m= 2l-1, \, l \in \mathbb{N}, \\
                   0 \quad &\quad  \makebox{if} \, \,  m= 2l, \, l \in \mathbb{N}.
                           \end{aligned} \right.$$
Put $$T=\left(
                        \begin{array}{cc}
                          1 & 0  \\
                          0 & S \\
                        \end{array}
                      \right) \makebox{\ on \ } H=\mathbb{C}e_{1} \oplus
                      l_{2}(\mathbb{N}).
$$ Then a simple calculation shows that $T^{m}= \left(
                        \begin{array}{cc}
                          1 & 0  \\
                          0 & 0 \\
                        \end{array}
                      \right) $ for all $m \geq 2$ and  $T^{*}T=\left(
                        \begin{array}{cc}
                          1 & 0  \\
                          0 & P \\
                        \end{array}
                      \right) \makebox{ on } H=\mathbb{C}e_{1} \oplus
                      l_{2}(\mathbb{N}),$ where $P$ is the
                      projection onto the closed subspace spanned by $\{e_{1},e_{3}, e_{5}, \cdots
                      \}.$ Hence $T$ is normaloid because $||T^{m}||=||T||^{m}=1$.
We claim that $T$ is not $(n,1)$-quasi-$*$-paranormal, for all
nonnegative integer $n$.
                      In fact, Lemma \ref{2.2} shows that $T$ is $(n,1)$-quasi-$*$-paranormal
                      if and only if
$$T^{*} T^{*(1+n)}T^{1+n}T - (1+n) \mu^{n} T^{*}TT^{*}T + n
 \mu^{1+n}T^{*}T \geq 0$$
 for all $\mu > 0$. Putting $\mu =1$, we obtain
\begin{align*}  & \qquad T^{*} T^{*(1+n)}T^{1+n}T - (1+n)T^{*}TT^{*}T + nT^{*}T \\
        & \,\,\, = \left(\begin{array}{cc}
                          1 & 0  \\
                          0 & 0 \\
                        \end{array}
                      \right)  -(1+n)\left(\begin{array}{cc}
                          1 & 0  \\
                          0 & P \\
                        \end{array}
                      \right)+ n \left(\begin{array}{cc}
                          1 & 0  \\
                          0 & P \\
                        \end{array}
                      \right) = \left(\begin{array}{cc}
                          0 & 0  \\
                          0 & -P \\
                        \end{array}
                      \right)  < 0.
     \end{align*}
Hence $T$ is not $(n,1)$-quasi-$*$-paranormal.

\end{upshape}
\end{example}

\section{A matrix representation}

\quad\,The following observation is a structure property for
$(n,k)$-quasi-$*$-paranormal operators.

\begin {observation} \label{3.1}
Suppose that $T^{k}H$ is not dense. Let $$ T= \left(
                                              \begin{array}{cc}
                                                T_{1} &  T_{2} \\
                                                0 &  T_{3} \\
                                              \end{array}
                                            \right)
\makebox{ on } H = \overline{T^{k}H} \oplus \mathrm{ker}(T^{*k}) $$
where $\overline{T^{k}H}$ is the closure of $T^{k}H$. If $T$ is
$(n,k)$-quasi-$*$-paranormal, then $T_{1}$ is $n$-$*$-paranormal,
$T_{3}^{k}=0$ and $\sigma(T)=\sigma(T_{1}) \cup \{0\}.$

\end{observation}

\begin{proof} Since $T_{1}^{1+n}z=T^{1+n}z$ for all $z \in
\overline{T^{k}H}$, $T_{1}$ is $n$-$*$-paranormal. Let $x \in
\mathrm{ker}(T^{*k})$. Then
$$ T^{k}x= \left(
                                              \begin{array}{cc}
                                                T_{1}^{k} &  \sum_{i=0}^{k-1}T_{1}^{i}T_{2}T_{3}^{k-1-i} \\
                                                0 &  T_{3}^{k} \\
                                              \end{array}
                                            \right)(0\oplus x) \in
                                            \overline{T^{k}H}. $$
Hence $T_{3}^{k}=0$ and $\sigma(T)=\sigma(T_{1}) \cup \{0\}.$
\end{proof}

The above matrix representation of $(n,k)$-quasi-$*$-paranormal
operators motivates the following

\begin {question} \label{3.2} Let $H, K$ be two
infinite dimensional complex Hilbert spaces. If $A$ is
$n$-$*$-paranormal and $C^{k}=0$, then the operator matrix $$ T=
\left(
                                              \begin{array}{cc}
                                                A &  B \\
                                                0 &  C \\
                                              \end{array}
                                            \right)$$ acting on $H \oplus K$ is
                                            $(n,k)$-quasi-$*$-paranormal $?$
\end{question}

Before giving a negative answer to this question, we present the
following

\begin {theorem} \label{3.3}
 Let $T$ be an operator on $H \oplus K$ defined as $$ T= \left(
                                              \begin{array}{cc}
                                                A &  B \\
                                                0 &  C \\
                                              \end{array}
                                            \right).$$
If $A$ is $n$-$*$-paranormal and surjective and $C^{k}=0$, then $T$
is similar to an $(n,k)$-quasi-$*$-paranormal operator.

\end{theorem}

\begin{proof}  Since $A$ is surjective and $C^{k}=0$, we have $\sigma_{s}(A) \cap \sigma_{a}(C) =
\varnothing$, where $\sigma_{s}(\cdot)$ and $\sigma_{a}(\cdot)$
denote the surjective spectrum and the approximative point spectrum
respectively. It then follows from part (c) of Theorem 3.5.1 in
[\cite{Laursen-Neumann}] that there exist some $S \in L(K,H)$ for
which $AS-SC=B$. Since
$$\left(
                                              \begin{array}{cc}
                                                I &  S \\
                                                0 &  I \\
                                              \end{array}
                                            \right) \left(
                                              \begin{array}{cc}
                                                A &  B \\
                                                0 &  C \\
                                              \end{array}
                                            \right) = \left(
                                              \begin{array}{cc}
                                                A &  0 \\
                                                0 &  C \\
                                              \end{array}
                                            \right)\left(
                                              \begin{array}{cc}
                                                I &  S \\
                                                0 &  I \\
                                              \end{array}
                                            \right),$$
it follows that $T$ is similar to $R:= \left(
                                              \begin{array}{cc}
                                                A &  0 \\
                                                0 &  C \\
                                              \end{array}
                                            \right).$
Let $x= x_{1} \oplus x_{2} \in H \oplus K$. Since $A$ is
$n$-$*$-paranormal and $C^{k}=0$, we have
  \begin{align*} \qquad \qquad
      ||R^*(R^{k}x)|| & =  ||R^*(R^{k}(x_{1} \oplus x_{2}))|| = ||A^*(A^{k}x_{1})||
        \\ & \leq
       ||A^{1+n}(A^{k}x_{1})||^{1/(1+n)}||A^{k}x_{1}||^{n/(1+n)} \\ & = ||R^{1+n}(R^{k}(x_{1} \oplus x_{2}))||^{1/(1+n)}||R^{k}(x_{1} \oplus x_{2})||^{n/(1+n)}
       \\ & = ||R^{1+n}(R^{k}x)||^{1/(1+n)}||R^{k}x||^{n/(1+n)}.
     \end{align*}
  Thus $T$ is similar to an $(n,k)$-quasi-$*$-paranormal operator.
\end{proof}

The following simple example provides a negative answer to Question
\ref{3.2} and a question posed by Mecheri [\cite{Mecheri-SM}]: Is
the operator matrix $$ T= \left(
                                              \begin{array}{cc}
                                                A &  B \\
                                                0 &  C \\
                                              \end{array}
                                            \right)$$
acting on $H \oplus K$ is $k$-quasi-$*$-class $A$ if $A$ is
$*$-class $A$ and $C^{k}=0$? It also shows us that
[\cite{Mecheri-SM}, Theorem 2.1] is not correct.

\begin{example} \label{3.4} \begin{upshape} Let $l_{2}(\mathbb{N})$ with the canonical orthogonal basis
$\{e_{m}\}_{m=1}^{\infty}$ and put $$T=\left(
                        \begin{array}{cc}
                          A & B  \\
                          0 & C \\
                        \end{array}
                      \right)= \left(
                        \begin{array}{cc}
                          0 & I  \\
                          0 & 0 \\
                        \end{array}
                      \right) \makebox{\ on \ } H= l_{2}(\mathbb{N}) \oplus
                      l_{2}(\mathbb{N}).
$$ Obviously, $A=0$ is $n$-$*$-paranormal (respectively, $*$-class $A$) for all nonnegative integer $n$ and
$C^{k}=0^{k}=0$, since we know that $k \geq 1$ from the assumption.
However, $T$ is not $(n,k)$-quasi-$*$-paranormal for all nonnegative
integer $n$, since
 \begin{align*} \qquad \qquad
      & \qquad ||T^{*}(e_{1}\oplus 0)||= ||\left(
                        \begin{array}{cc}
                          0 & 0  \\
                          I & 0 \\
                        \end{array}
                      \right)(e_{1}\oplus 0)|| = ||0 \oplus e_{1}||
        \\ & \,\,\, = 1 > 0 =||T^{1+n+k}(e_{1}\oplus 0)||^{1/(1+n)}||T^{k}(e_{1}\oplus
                      0)||^{n/(1+n)}.
     \end{align*}
 In particular, $T$ is not $(1,k)$-quasi-$*$-paranormal. Hence by Theorem \ref{2.1}(1), $T$ is not $k$-quasi-$*$-class
 $A$.

\end{upshape}
\end{example}

\section{Joint (approximate) point spectrum and SVEP}

\quad\,[\cite{Uchiyama}, Theorem] illustrated that the following
result is not true even for paranormal operators.

\begin {theorem} \label{4.1}
 Let $T$ be $(n,k)$-quasi-$*$-paranormal and $0 \neq \lambda \in
 \mathbb{C}$.

 $(1)$ If $(T-\lambda)x=0,$ then $(T^{*} - \overline{\lambda})x=0.$
 Consequently
 $$\mathrm{ker}(T-\lambda) \subseteq \mathrm{ker}(T^{*}-\overline{\lambda}).$$

 $(2)$ If $(T-\lambda)x_{m} \rightarrow 0$ for a sequence $\{x_{m}\}_{m=1}^{\infty}$ of unit vectors, then $(T^{*} -
 \overline{\lambda}) x_{m} \rightarrow 0$.
\end{theorem}

Clearly, if $\lambda =0$, then the above properties hold for
$(n,0)$-quasi-$*$-paranormal (that is $n$-$*$-paranormal) operators.
However [\cite{Tanahashi-Uchiyama}, Example 6] showed that, when
$\lambda =0$, the above properties do not hold even for
quasi-hyponormal (that is, $(0,1)$-quasi-$*$-paranormal) operators.

\begin{proof}
(1) We may assume that $\lambda \in \sigma_{p}(T)$, where
$\sigma_{p}(T)$ is the point spectrum of $T$.
 Let $x \in \mathrm{ker}(T-\lambda)$ and $||x||=1$. Then
 $$ ||T^{*}(T^{k}x)||^{n+1} \leq ||T^{1+n}(T^{k}x)||||T^{k}x||^{n},$$
 hence $$ |\lambda|^{k(n+1)}||T^{*}x||^{n+1} \leq |\lambda|^{k(n+1)}|\lambda|^{n+1}.$$
 That is,
 $$||T^{*}x|| \leq |\lambda|.$$
 Thus we have
 \begin{align*} \qquad \qquad  \qquad
      ||T^{*}x- \overline{\lambda}x||^{2}  &=  (T^{*}x- \overline{\lambda}x, T^{*}x- \overline{\lambda}x) \\ &=
      (T^{*}x,T^{*}x)-\overline{\lambda}(x,T^{*}x)-\lambda(T^{*}x,x)+|\lambda|^{2}  \\
        & =||T^{*}x||^{2} - \overline{\lambda}(Tx,x)-\lambda(x,Tx)+|\lambda|^{2}  \\  & \leq |\lambda|^{2} - |\lambda|^{2}-|\lambda|^{2}+|\lambda|^{2}=0.
     \end{align*}
 Hence $||T^{*}x- \overline{\lambda}x||^{2} \leq 0 $ and
 consequently $x \in \mathrm{ker}(T^{*}-\overline{\lambda}).$

 (2) Let $(T-\lambda)x_{m} \rightarrow 0$ for unit vectors
 $\{x_{m}\}$ and let $l \in \mathbb{N}$. Since $$T^{l}=(T-\lambda+\lambda)^{l}= \sum\limits_{j=1}^{l} \binom{l}{j}
 \lambda^{l-j}(T-\lambda)^{j}+\lambda^{l},
  $$ we have $(T^{l}-\lambda^{l})x_{m}\rightarrow 0.$ It then follows
  form \begin{align*}
        \big{|}||\lambda^{l}x_{m}|| -
  ||(T^{l}-\lambda^{l})x_{m}||\big{|} \leq  ||T^{l}x_{m}||  & = || \lambda^{l}x_{m} + (T^{l}-\lambda^{l})x_{m}||  \\ &  \leq  ||\lambda^{l}x_{m}|| +
||(T^{l}-\lambda^{l})x_{m}||
     \end{align*}
 that $||T^{l}x_{m}|| \rightarrow |\lambda|^{l}.$ In particular, we have
 \begin{equation} \label{eq 4.1} \qquad \qquad\qquad
||T^{1+n+k}x_{m}|| \rightarrow |\lambda|^{1+n+k} \makebox{ and }
||T^{k}x_{m}|| \rightarrow
  |\lambda|^{k}.
 \end{equation} Moreover,
  \begin{equation} \label{eq 4.2} \qquad \qquad\qquad
\big{|}||T^{*}\lambda^{k}x_{m}|| - ||T^{*}(T^{k}
-\lambda^{k})x_{m}|| \big{|}  \leq ||T^{*}T^{k}x_{m}||.
 \end{equation}
 Since $T$ is $(n,k)$-quasi-$*$-paranormal,
 $$ ||T^{*}(T^{k}x_{m})||^{n+1} \leq ||T^{1+n}(T^{k}x_{m})||||T^{k}x_{m}||^{n}. $$
 Then it follows form (\ref{eq 4.1}) and (\ref{eq 4.2}) that
 $$ \limsup\limits_{m \rightarrow \infty}||T^{*}x_{m} || \leq |\lambda|.$$
Since
 \begin{align*}
      ||T^{*}x_{m}- \overline{\lambda}x_{m}||^{2}  &=  (T^{*}x_{m}- \overline{\lambda}x_{m}, T^{*}x_{m}- \overline{\lambda}x_{m}) \\ &=
      (T^{*}x_{m},T^{*}x_{m})-\overline{\lambda}(x_{m},T^{*}x_{m})-\lambda(T^{*}x_{m},x_{m})+|\lambda|^{2}  \\
        & =||T^{*}x_{m}||^{2} - \overline{\lambda}(Tx_{m},x_{m})-\lambda(x_{m},Tx_{m})+|\lambda|^{2}
        \\  & =||T^{*}x_{m}||^{2} -
        \overline{\lambda}((T-\lambda)x_{m},x_{m})-\lambda(x_{m},(T-\lambda)x_{m})-|\lambda|^{2},
     \end{align*}
we have $$ \limsup\limits_{m \rightarrow \infty}||T^{*}x_{m}-
\overline{\lambda}x_{m}||^{2} \leq |\lambda|^{2}-|\lambda|^{2}=0.$$
 This establishes that $(T^{*}- \overline{\lambda})x_{m} \rightarrow 0.$
\end{proof}

\begin {remark} \label{4.2} \begin{upshape} We note that, with the Berberian faithful $*$-representation [\cite{Berberian}] at hand,
part (2) of Theorem \ref{4.1} can also be deduced from part (1) of
it.
\end{upshape}
\end{remark}

For $T \in L(H)$, let $\sigma_{p}(T)$, $\sigma_{jp}(T)$,
$\sigma_{a}(T)$ and $\sigma_{ja}(T)$ denote the point spectrum,
joint point spectrum, approximate point spectrum and joint
approximate point spectrum of $T$, respectively (see
[\cite{Aluthge-Wang}]). Many mathematicians shown that, for some
nonhyponormal operators $T$, the nonzero points of its point
spectrum and joint point spectrum are identical, the nonzero points
of its approximate point spectrum and joint approximate point
spectrum are identical (see [\cite{Aluthge-Wang,
Uchiyama-Tanahashi-Lee, Yang-Yuan, Yuan-Gao}]). The following
corollary, which is an immediate consequence of Theorem \ref{4.1},
extends this result to the class of $(n,k)$-quasi-$*$-paranormal
operators.

\begin {corollary} \label{4.3} If $T$ is $(n,k)$-quasi-$*$-paranormal,
then $$\sigma_{jp}(T) \backslash \{0\}=\sigma_{p}(T)\backslash \{0\}
\makebox{ and } \sigma_{ja}(T)\backslash
\{0\}=\sigma_{a}(T)\backslash \{0\}.$$
\end{corollary}

In the next example, we invoke Uchiyama's example again to
illustrate that the above equalities do not hold even for paranormal
operators.

\begin {example} \label{4.4} \begin{upshape} Recall that Uchiyama
have constructed in [\cite{Uchiyama}, Theorem] an operator defined
by
$$T=\left(
                        \begin{array}{cc}
                          1 & \alpha e_{1} \otimes e_{1}  \\
                          0 & U+1 \\
                        \end{array}
                      \right) \makebox{\ on \ } H=\mathbb{C}e_{1} \oplus
                      l_{2}(\mathbb{N}),
$$ where $U$ is the unilateral right shift operator on
$l_{2}(\mathbb{N})$ with the canonical orthogonal basis
$\{e_{n}\}_{n=1}^{\infty}$. He proved that $T$ is paranormal for $0
< \alpha < \frac{1}{4}$ and $\mathrm{ker}(T-1)=\mathbb{C}e_{1}
\oplus \{0\}$. A step further, we have
$\mathrm{ker}(T^{*}-1)=\{ae_{1} \oplus (be_{1} + a\alpha e_{2}): a,b
\in \mathbb{C} \}$, hence
$$\mathrm{ker}(T-1) \cap \mathrm{ker}(T^{*}-1) =\{0\}.$$
Consequently, $1 \in \sigma_{p}(T) \backslash  \sigma_{jp}(T).$
Evidently, $1 \in \sigma_{a}(T)$. Next, we show that $1 \notin
\sigma_{ja}(T)$. Otherwise, there exists a sequence
$\{x_{n}\}_{n=1}^{\infty}$ of unit vectors satisfying $(T - 1)x_{n}
\rightarrow 0$ and $(T^{*} - 1)x_{n} \rightarrow 0$. For $n \in
\mathbb{N}$, let $x_{n}=a_{n}e_{1} \oplus (b_{1,n}, b_{2,n}, \cdots)
\in \mathbb{C}e_{1} \oplus l_{2}(\mathbb{N})$. Then
\begin{equation} \label{eq 4.3} \qquad \qquad\qquad\qquad\qquad
a_{n}^{2}+ \sum\limits_{k=1}^{\infty}b_{k,n}^{2} =1,
 \end{equation}
 \begin{equation} \label{eq 4.4} \qquad \qquad\qquad\qquad\qquad
\alpha b_{1,n}^{2}+ \sum\limits_{k=1}^{\infty}b_{k,n}^{2}
\rightarrow 0,
 \end{equation}
 and
\begin{equation} \label{eq 4.5} \qquad \qquad\qquad\qquad\qquad
(\alpha a_{n} + b_{2,n})^{2}+ \sum\limits_{k=3}^{\infty}b_{k,n}^{2}
\rightarrow 0.
 \end{equation}
 By (\ref{eq 4.4}), we have $\sum\limits_{k=1}^{\infty}b_{k,n}^{2}
\rightarrow 0$, $\sum\limits_{k=2}^{\infty}b_{k,n}^{2} \rightarrow
0$ and $b_{2,n} \rightarrow 0.$ Then by (\ref{eq 4.3}), we have
$a_{n}^{2} \rightarrow 1$. Thus, we have $(\alpha a_{n} +
b_{2,n})^{2}+ \sum\limits_{k=3}^{\infty}b_{k,n}^{2}=\alpha^{2}
a_{n}^{2}+2\alpha a_{n}b_{2,n}+\sum\limits_{k=2}^{\infty}b_{k,n}^{2}
\rightarrow \alpha^{2}$, which contradicts to (\ref{eq 4.5}).

\end{upshape}
\end{example}

\begin {corollary} \label{4.5} If $T$ is $(n,k)$-quasi-$*$-paranormal and $\lambda \neq \mu$,
then $$\mathrm{ker}(T-\lambda) \, \bot \, \mathrm{ker}(T-\mu).$$
\end{corollary}

\begin{proof} Without loss of generality, we may suppose that $\mu \neq 0$. Let $x \in \mathrm{ker}(T-\lambda)$ and $y \in
\mathrm{ker}(T-\mu)$. Then by Theorem \ref{4.1}(1), we have
$$\lambda(x,y)=(Tx,y)=(x,T^{*}y)=(x,\overline{\mu}y)=\mu(x,y),$$
which implies $(x,y)=0$ and so $\mathrm{ker}(T-\lambda)\, \bot \,
\mathrm{ker}(T-\mu)$.
\end{proof}

\begin {corollary} \label{4.6} If $T$ is $(n,k)$-quasi-$*$-paranormal,
then $\mathrm{ker}(T^{1+k})=\mathrm{ker}(T^{2+k})$ and
$\mathrm{ker}(T-\lambda)=\mathrm{ker}(T-\lambda)^{2}$ when $0 \neq
\lambda \in \mathbb{C}$.
\end{corollary}

\begin{proof} Since $T$ is $(n,k)$-quasi-$*$-paranormal, it is $(n,k+1)$-quasi-$*$-paranormal by Theorem
\ref{2.1}(4), and hence it is $(n+1,k)$-quasiparanormal by Theorem
\ref{2.1}(2). Therefore
$\mathrm{ker}(T^{n+k+2})=\mathrm{ker}(T^{1+k})$ and so
$\mathrm{ker}(T^{1+k})=\mathrm{ker}(T^{2+k})$. When $0 \neq \lambda
\in \mathbb{C}$, we need only to show that $\mathrm{ker}(T-\lambda)
\supseteq \mathrm{ker}(T-\lambda)^{2}$, since
$\mathrm{ker}(T-\lambda) \subseteq \mathrm{ker}(T-\lambda)^{2}$ is
clear. Theorem \ref{4.1}(1) implies that
$(T-\lambda)^{*}(T-\lambda)\mathrm{ker}(T-\lambda)^{2}=\{0\}$, hence
$(T-\lambda)\mathrm{ker}(T-\lambda)^{2} \subseteq
\mathrm{ker}(T-\lambda)^{*} \cap (T-\lambda)H=\{0\}$ and so
$\mathrm{ker}(T-\lambda)^{2} \subseteq \mathrm{ker}(T-\lambda)$.
\end{proof}

An operator $T \in L(H)$ is said to be have single valued extension
property at $\lambda_{0} \in \mathbb{C}$ (SVEP at $\lambda_{0}$ for
brevity) if for every open neighborhood $\mathcal {U}$ of
$\lambda_{0}$, the only analytic function $f:\mathcal {U}
\rightarrow H$ which satisfies the equation $(\lambda I-T)f(\lambda)
= 0$ for all $\lambda \in \mathcal {U}$ is the constant function $f
\equiv 0$. Let $\mathcal {S}(T):=\{\lambda \in \mathbb{C}:T
\makebox{ does not have the SVEP at } \lambda \}$. An operator $T
\in L(H)$ is said to be have SVEP if $\mathcal {S}(T) =
\varnothing$.

\begin {corollary} \label{4.7} If $T$ is $(n,k)$-quasi-$*$-paranormal,
then $T$ has \begin{upshape}SVEP\end{upshape}.
\end{corollary}

\begin{proof} It follows directly form Corollary \ref{4.5} and [\cite{Yuan-Ji},
Lemma 3.5] (or, Corollary \ref{4.6} and [\cite{Laursen}, Proposition
1.8]).
\end{proof}

SVEP, which is an important property in local spectral theory and
Fredholm theory, has a number of consequences. We list in the next
corollary some of these.

For $T \in L(H)$, let $\sigma_{W}(T)$, $\sigma_{SF_{+}^{-}}(T)$,
$\sigma_{BW}(T)$, $\sigma_{SBF_{+}^{-}}(T)$ and $\sigma_{LD}(T)$
denote the Weyl spectrum, upper semi-Weyl spectrum, B-Weyl spectrum,
upper semi-B-Weyl spectrum and left Drazin spectrum of $T$,
respectively (see [\cite{Berkani-Koliha,Zeng-Zhong}]).

We say that $T \in L(H)$ is algebraically
$(n,k)$-quasi-$*$-paranormal, if there exist a nonconstant complex
polynomial $p$ such that $p(T)$ is $(n,k)$-quasi-$*$-paranormal. Let
$H(\sigma(T))$ denote the space of all functions analytic on some
open neighborhood $\mathcal {U}$ containing $\sigma(T)$. If $T$ has
SVEP then so does $f(T)$ for all $f \in H(\sigma(T))$; conversely,
if $p(T)$ has SVEP for some nonconstant polynomial $p$ then $T$ does
([\cite{Aiena}, Theorem 2.40]). Hence algebraically
$(n,k)$-quasi-$*$-paranormal operators have SVEP by Corollary
\ref{4.7}.

\begin {corollary} \label{4.8} Let $f \in H(\sigma(T))$. If $T$ or $T^{*}$ is algebraically
$(n,k)$-quasi-$*$-paranormal, then

$(1)$ $f(T)$ and $f(T^{*})$ obey to $a$-Browder's theorem.

$(2)$ $f(T)$ possesses property $(gb)$, when $T^{*}$ is
algebraically $(n,k)$-quasi-$*$-paranormal.

$(3)$ $\sigma_{W}(f(T))= f(\sigma_{W}(T))$ and analogous equality
holds for $\sigma_{SF_{+}^{-}}(\cdot)$.

$(4)$ if $f$ is nonconstant on each component of open neighborhood
$\mathcal {U}$ containing $\sigma(T)$, then $\sigma_{BW}(f(T))=
f(\sigma_{BW}(T))$ and analogous equality holds for
$\sigma_{SBF_{+}^{-}}(\cdot)$.
\end{corollary}

\begin{proof} (1), (2) and (3) follow form [\cite{Aiena}, Corollary 3.73],
[\cite{Zeng-Zhong}, Theorem 2.14] and [\cite{Aiena}, Corollary
3.72], respectively.

(4) By [\cite{Amouch-Zguitti}, Theorem 2.4], it remains to show that
$\sigma_{SBF_{+}^{-}}(f(T))= f(\sigma_{SBF_{+}^{-}}(T))$. From
[\cite{Amouch-Zguitti}, Theorem 2.1] and [\cite{Zeng-Zhong}, Lemma
2.3], it follows that $\sigma_{SBF_{+}^{-}}(T) \cup (\mathcal {S}(T)
\cap \mathcal {S}(T^{*})) = \sigma_{LD}(T)$ for any $T \in L(H)$.
And then the assertion follows form the fact that left Drazin
spectrum satisfies the spectral mapping theorem for such a function
$f$.
\end{proof}

{\bf Acknowledgements.} The authors express their sincere thanks to
Professor Jiangtao Yuan for sending us the papers
[\cite{Uchiyama-Tanahashi-Lee,Yang-Yuan,Yuan-Ji}].



\end{document}